\documentclass[12pt]{amsart}
\usepackage{amssymb,amsmath,amsthm,amscd,epsfig,mathrsfs}
\usepackage[backref]{hyperref}
\usepackage[normalem]{ulem}
\usepackage[all,cmtip]{xy}
\usepackage{ stmaryrd }

\newtheorem{thm}{Theorem}[section]
\newtheorem{lem}[thm]{Lemma}
\newtheorem{cor}[thm]{Corollary}

\theoremstyle{definition}

\newtheorem{question}[thm]{Question}

\newtheorem{example}[thm]{Example}

\theoremstyle{remark}
\newtheorem{rmk}[thm]{Remark}

\usepackage[usenames,dvipsnames]{color}

\def\Z{\ifmmode{{\mathbb Z}}\else{${\mathbb Z}$}\fi}
\def\Q{\ifmmode{{\mathbb Q}}\else{${\mathbb Q}$}\fi}
\def\C{\ifmmode{{\mathbb C}}\else{${\mathbb C}$}\fi}
\def\P{\ifmmode{{\mathbb P}}\else{${\mathbb P}$}\fi}
\def\H{\ifmmode{{\mathrm H}}\else{${\mathrm H}$}\fi}
\def\G{\ifmmode{{\mathbb G}}\else{${\mathbb G}$}\fi}
\def\R{\ifmmode{{\mathbb R}}\else{${\mathbb R}$}\fi}
\def\F{\ifmmode{{\mathbb F}}\else{${\mathbb F}$}\fi}
\def\O{\ifmmode{{\cal O}}\else{${\cal O}$}\fi}
\def\D{\ifmmode{{\cal{D}}^b}\else{${{\cal{D}}^b}$}\fi}

\newcommand{\Xbar}{{\overline{X}}}

\newcommand{\calA}{{\mathcal A}}
\newcommand{\calB}{{\mathcal B}}
\newcommand{\calC}{{\mathcal C}}

\newcommand{\calO}{{\mathcal O}}

\newcommand{\calX}{{\mathcal X}}

\DeclareMathOperator{\Char}{char}
\DeclareMathOperator{\inv}{inv}

\DeclareMathOperator{\im}{im}

\DeclareMathOperator{\Pic}{Pic}

\DeclareMathOperator{\Spec}{Spec}

\DeclareMathOperator{\Proj}{Proj}

\DeclareMathOperator{\et}{\acute{e}t}

\DeclareMathOperator{\cores}{cores}

\DeclareMathOperator{\id}{id}

\usepackage{todonotes}

\DeclareMathOperator{\CH}{CH}

\newcommand{\A}{\mathbb{A}}

\DeclareMathOperator{\BM}{BM}

\newcommand{\isom}{\simeq}

\newcommand{\Br}{\mathrm{Br}}
    \DeclareFontFamily{U}{wncy}{}
    \DeclareFontShape{U}{wncy}{m}{n}{<->wncyr10}{}
    \DeclareSymbolFont{mcy}{U}{wncy}{m}{n}
    \DeclareMathSymbol{\Sha}{\mathord}{mcy}{"58} 

\begin{document}

\title[Brauer classes that never obstruct the Hasse principle]{Index of fibrations and Brauer classes that never obstruct the Hasse principle}
\author{Masahiro Nakahara}
\address{Department of Mathematics, Rice University, 6100 Main St 
Houston, TX, 77054}
\email{mn24@rice.edu}
\urladdr{}
\date{\today}

\begin{abstract}
Let $X$ be a smooth projective variety with a fibration into varieties that either satisfy a condition on representability of zero-cycles or that are torsors under an abelian variety. We study the classes in the Brauer group that never obstruct the Hasse principle for $X$. We prove that if the generic fiber has a zero-cycle of degree $d$ over the generic point, then the Brauer classes whose orders are prime to $d$ do not play a role in the Brauer--Manin obstruction. As a result we show that the odd torsion Brauer classes never obstruct the Hasse principle for del Pezzo surfaces of degree 2, certain K3 surfaces, and Kummer varieties.
\end{abstract}

\maketitle

\section{Introduction}

Let $X$ be a smooth projective geometrically integral variety over a number field $k$. Let $\A$ denote the adeles of $k$ and $X(\A)$ be the set of adelic points on $X$. Following Manin \cite{manin}, one can use any subset $H$ of the Brauer group $\Br(X):=\H^2_{\et}  (X,\G_m)_{\text{tors}}$ to form a subset $X(\A)^H\subseteq X(\A)$, which contains $X(k)$ by class field theory. This gives an inclusion reversing map
\begin{equation*}
\begin{array}{c @{{}\ {}} c @{{}\ {}} c}
\{\text{subsets of }\Br(X)\}&\longrightarrow &\{\text{subsets of }X(\A)\},\\
H&\longmapsto &X(\A)^H.
\end{array}
\end{equation*}
See \S\ref{sec:mainth} for details. Hence, any subset $H\subseteq\Br(X)$ gives the containments
\begin{equation}\label{contain}
X(\A)^\Br:=X(\A)^{\Br(X)}\subseteq X(\A)^H\subseteq X(\A).
\end{equation}
For any positive integer $n$, define the set of prime-to-$n$ torsion elements by
$$\Br(X)[n^\perp]:=\{\calA\in\Br(X)\mid m\calA=0\text{ for some } (n,m)=1\}.$$
We consider the following general question.

\begin{question}\label{q} Given a class $\calC$ of $k$-varieties and an integer $n>1$, does $X(\A)\neq\emptyset$ imply $X(\A)^{\Br(X)[n^\perp]}\neq\emptyset$ for all $X\in \calC$? In other words, is it true that $\Br(X)[n^\perp]$ never obstructs the Hasse principle?
\end{question}

In this paper, we study some nontrivial cases where the answer to Question \ref{q} is positive. Denote by $\Br_0(X)$ the subgroup in $\Br(X)$ of constant classes $\im[\Br(k)\to\Br(X)]$, arising from the structure morphism $X\to\Spec(k)$. If $\Br(X)[n^\perp]=\Br_0(X)[n^\perp]$ for every $X\in\calC$, then the answer to Question \ref{q} will be trivially positive. Hence a particular case of interest is when $n$ is chosen so that $\Br(X)[n^\perp]\supsetneq\Br_0(X)[n^\perp]$ for some $X\in\calC$. For example, we consider the class of degree $d$ del Pezzo surfaces, where the possible Brauer groups are known.

\begin{example}\label{deg3ex}
Let $\calC$ be the class of del Pezzo surfaces of degree 3 over $k$. Then $\Br(X)/\Br_0(X)$ has exponent 2 or 3 for any $X\in \calC$ \cite{swd93}. Moreover if  $\Br(X)/\Br_0(X)$ has exponent 2, then $X$ satisfies the Hasse principle \cite[Corollary 1]{swd93}. Hence Question \ref{q} has a positive answer for $n=3$.
\end{example}

In general, the possible Brauer groups for a class $\calC$ of varieties are difficult to compute; in particular it is already a difficult problem to find an $n$ such that $\Br(X)[n^\perp]=\Br_0(X)[n^\perp]$ for all $X\in\calC$.

Some recent results on Brauer--Manin obstructions are related to the framework of Question \ref{q}. For example, in \cite{is2015}, Ieronymou and Skorobogatov show that for the class $\calC$ of diagonal quartic surfaces in $\P^3_\Q$, the odd torsion part $\Br(X)[2^\perp]$ of the Brauer group does not obstruct the Hasse principle, i.e., if $X(\A_\Q)\neq\emptyset$ then $ X(\A_\Q)^{\Br[2^\perp]}\neq\emptyset$. This gives an answer to Question \ref{q} for diagonal quartics over $\Q$ with $n=2$. They also give conditions on the equation of $X\in\calC$ for the group $\Br(X)/\Br_0(X)$ to contain odd torsion elements, so in particular $\Br(X)[2^\perp]\supsetneq\Br_0(X)[2^\perp]$. In \cite{cv17}, Creutz and Viray consider a related but generally logically independent question. For the class $\calC$ of degree $d$ $k$-varieties in projective space, they look at the leftmost containment of \eqref{contain} and consider whether $X(\A)^{\Br[d^\infty]}\neq\emptyset\implies X(\A)^{\Br}\neq\emptyset$ holds. They write $\BM_d$ for this implication and prove it holds for some classes of varieties such as torsors under abelian varieties and Kummer varieties. This implication is false in general \cite[Theorem 6.5]{cv17}. They also write $\BM_d^\perp$ for the statement there is no prime-to-$d$ Brauer--Manin obstruction. A positive answer to Question \ref{q} for $n=d$ is equivalent to $\BM^\perp_{d}$.

Inspired by Creutz--Viray, one can also ask for a stronger version of Question \ref{q} of whether both $\BM_n$ and $\BM^\perp_{n}$  hold simultaneously,

\begin{question}\label{q2} Given a class $\calC$ of $k$-varieties and an integer $n>1$, does $X(\A)^B\neq\emptyset$ imply $X(\A)^{B+\Br(X)[n^\perp]}\neq\emptyset$ for all $X\in \calC$ and all subgroups $B\in \Br(X)$? In other words, is it true that $\Br(X)[n^\perp]$ plays no role in obstructing the Hasse principle?
\end{question}

Question \ref{q} is a special case of Question \ref{q2} where $B$ is trivial.

\subsection{Main results} For a smooth projective variety $Y$ over a number field $k$, let $A_0(Y)$ be the group of zero-cycles of degree zero, modulo rational equivalence. We say that $Y$ satisfies property (ZC) if for any field extension $K/k$ and $Q\in Y(K)$, the natural map $Y(K)\to A_0(Y_{K})$ given by $P\mapsto (P)-(Q)$ is surjective. For example, smooth projective curves of genus 1 and $k$-rational varieties satisfy (ZC) (see \ref{propp}, \ref{ppp3} for more examples). Define the index of a variety $Y$ to be the smallest positive integer $d$ such that $Y$ has a zero-cycle of degree $d$. We say that $Y$ satisfies weak approximation if the image of $Y(k)$ is dense in $Y(\A)$ with the product topology.

We prove the following theorem, which we use to study Question \ref{q2} (and hence also Question \ref{q}) over any number field for all degree 2 del Pezzo surfaces and some diagonal quartic surfaces.

\begin{thm}\label{main} Let $\pi\colon X\to Z$ be a morphism between smooth projective geometrically integral varieties over a number field $k$. Suppose that $Z$ satisfies weak approximation and there exists a Zariski open set $Z_0\subseteq Z$ such that the fiber $X_P$ satisfies (ZC) for any $P\in Z_0$. Suppose that the generic fiber over the function field $k(Z)$ has index $d$.  If $B\subset \Br(X)$ is a subgroup such that $X(\A)^B\neq\emptyset$, then $X(\A)^{B+\Br(X)[d^\perp]}\neq\emptyset$.
\end{thm}

We also prove a similar result when the fibers are torsors under abelian varieties. We use this to study Question \ref{q} for Kummer varieties.

\begin{thm}\label{main2} Let $\pi\colon X\to Z$ be a morphism between smooth projective geometrically integral varieties over a number field $k$. Suppose that $Z$ satisfies weak approximation. Suppose that the generic fiber $Y$ is a $k(Z)$-torsor under an abelian variety $A/k(Z)$, and that the order of $[Y]\in \H^1(k(Z),A)$ is $d$. If $B\subset \Br(X)$ is a subgroup such that $X(\A)^B\neq\emptyset$, then $X(\A)^{B+\Br(X)[d^\perp]}\neq\emptyset$.
\end{thm}

\begin{rmk} The case when $Z=\Spec k$ is a point was proven in \cite{cv17}.
\end{rmk}

\subsection{Applications}

Colliot-Th\'el\`ene and Sansuc \cite[\S V Question k1]{cts80} have conjectured that for geometrically rational surfaces, the Brauer--Manin obstruction is the only obstruction to the Hasse principle. In particular, it is known that del Pezzo surfaces of degree $d=1$ and $d\geq 5$ satisfy the Hasse principle (see, e.g., \cite{va13}). Some partial results for $d=3$ (see Example \ref{deg3ex}) and $d=4$ (see \cite{wbthesis} conditional on Schinzel's hypothesis and finiteness of Tate-Shafarevich groups) are known. In \cite[Question 4.5]{corndp2}, Corn asked whether there are any Brauer--Manin obstruction to rational points when $d=2$ and $\Br(X)/\Br(k)$ is 3 torsion. Moreover he asked if $X$ satisfies the Hasse principle in this case. If the conjecture is true, a negative answer to the first question would imply a positive answer to the second question.

Let $\calC$ be the class of del Pezzo surfaces of degree 2 over a number field $k$. A consequence of Theorem \ref{main} is that Question \ref{q2} has a positive answer with $n=2$. In other words, the 3-torsion part of the Brauer group will not obstruct the Hasse principle, giving a positive answer to one of Corn's questions. More concretely, in \S\ref{app} we prove the following result.

\begin{cor}\label{dP2} Let $X$ be a degree 2 del Pezzo surface over a number field $k$. Suppose that $\Br(X)/\Br_0(X)$ has exponent 3. Then $X(\A)\neq\emptyset\implies X(\A)^{\Br}\neq\emptyset$.
\end{cor}

\begin{rmk}
There exist minimal del Pezzo surfaces over $k$ with $\Br(X)/\Br_0(X)\isom \Z/3\Z$. Hence not every surface occurring in Corollary \ref{dP2} arises as a blow-up of a cubic surface at a rational point. See Example \ref{example}, which was communicated to us by Elsenhans.
\end{rmk}

Let $\calC$ be the class of all smooth diagonal quartics in $\P^3$ defined by
\begin{equation}\label{diag}
ax^4+by^4+cz^4+dw^4=0,
\end{equation}
where $a,b,c,d\in k^\times$ and $abcd\in k^{\times2}$ for some number field $k$. Then Question \ref{q2} has a positive answer with $n=2$. This extends a result of Ieronymou and Skorobogatov in \cite{is2015} to any number field, but under the condition that $abcd\in k^{\times2}$.

\begin{cor}\label{corodd} Let $X$ be a smooth diagonal quartic \eqref{diag} in $\P_k^3$, with $abcd \in k^{\times 2}$. If $B\subset \Br(X)$ is a subgroup such that $X(\A)^B\neq\emptyset$, then $X(\A)^{B+\Br(X)[2^\perp]}\neq\emptyset$.
\end{cor}

Given an abelian variety $A$ of dimesion $\geq2$ and a 2-covering of $A$, one can construct a Kummer variety attached to this 2-covering. Let $\calC$ be the class of all such Kummer varieites over a number field $k$. It was proven by Skorobogatov and Zarhin that Question \ref{q} has positive answer for $\calC$ with $n=2$ \cite{sz2016}, and later Skorobogatov extended the result to answer Question \ref{q2} in \cite[Theorem A.1]{cv17}. They're approach relied on proving results about $\Br(X)$ and its odd torsion part. We use the proof of Theorem \ref{main2} to give a more direct proof of this, but at the cost of not giving any information about $\Br(X)$ itself.

\begin{cor}\label{kum} Let $A$ be an abelian variety defined over a number field $k$. Let $X$ be the Kummer variety attached to a 2-covering of $A$ (see \S\ref{constkum}). If $B\subset \Br(X)$ is a subgroup such that $X(\A)^B\neq\emptyset$, then $X(\A)^{B+\Br(X)[2^\perp]}\neq\emptyset$.
\end{cor}

\section*{Acknowledgements}

I thank my advisor Anthony V\'{a}rilly-Alvarado for his continuous support and help. I thank Sho Tanimoto, Bianca Viray, and Olivier Wittenberg for valuable discussions and feedback. I thank Andreas-Stephan Elsenhans for providing the equation of the del Pezzo surface in Example \ref{example}. I thank Jean-Louis Colliot-Th\'el\`ene for his suggestions to a preliminary version of the paper, which helped improve the proof and scope of the original results. I thank Alexei Skorobogatov for his suggestion of using \cite[Lemma 4.6]{cv17} to prove Theorem \ref{main2}.

\section{Proof of Theorem \ref{main}}\label{sec:mainth}

Denote by $\Omega$ the set of places of $k$, and by $k_v$ the completion of $k$ at $v$ for any $v\in \Omega$. Given any field extension $K/k$, there is a pairing
\begin{equation*}
\begin{array}{c @{{}\ {}} c @{{}\ {}} c}
X(K)\times \Br(X) &\longrightarrow  &\Br(K),\\\
(P\ ,\ \calA) &\longmapsto &\ \calA (P):=\iota^*\calA.
\end{array}
\end{equation*}
where $i\colon P\to X$ is the inclusion. This leads to the definition of $X(\A)^H$,
$$X(\A)^H=\bigg\{\{P_v\}\in X(\A)\mid\sum_{v\in\Omega}\inv_v\calA (P_v)=0\ \forall\calA\in H\bigg\},$$
where $\inv_v\colon \Br(k_v)\to\Q/\Z$ is the invariant map from local class field theory. In \cite{ct95}, Colliot-Th\'el\`ene extended the above pairing to the group of zero-cycles on $X$. Moreover the pairing respects rational equivalence, thus inducing a pairing on $\CH_0(X_K)$ defined as follows
\begin{equation*}
\begin{array}{c @{{}\ {}} c @{{}\ {}} c}
\CH_0(X_K)\times \Br(X) &\longrightarrow  &\Br(K),\\\
\bigg(\displaystyle\sum_i n_i (P_i)\ ,\ \calA\bigg) &\longmapsto &\displaystyle\sum_in_i\cores_{K(P_i)/K}\calA (P_i).
\end{array}
\end{equation*}
Similarly, by class field theory, if $M\in\CH_0(X)$, its image $\{M_v\}\in \prod_{v\in\Omega}\CH_0(X_{k_v})$ then satisfies $\sum_v\inv_v \calA(M_v)=0$.

\begin{proof}[Proof of Theorem \ref{main}] By, e.g., \cite[Lemma 4.8]{cv17}, it suffices to show that $X(\A)^B\neq\emptyset\implies X(\A)^H\neq\emptyset$ for any finite subgroup $H\subset B+\Br(X)[d^\perp]$. Let $H \subset B+\Br(X)[d^\perp]$ be a finite subgroup. Then there is a finite set $S\subset \Omega$ of places in $k$ such that $X$ can be spread out to a scheme $\calX$ over $\Spec \calO_{k,S}$, where $\calO_{k,S}$ denotes the ring of $S$-integers, and any $\calA\in H$ extends to an element of $\Br(\calX)$ (see \cite[\S5]{skorobogatovtors}). As a result, the evaluation of any point $P_v\in X(k_v)$ will lie in $\Br(\calO_v)$, which is trivial; here $\calO_v$ denotes the ring of integers of $k_v$. Since the existence of a model over $\Spec \calO_{k,S}$ is stable under a finite field extension, the map $\calA(-)\colon X(L_w)\to\Br(L_w)$ is also zero for any finite extension $L_w/k_v$. Hence it follows from the definition that the map $\calA(-)\colon \CH_0(X_{k_v})\to \Br(k_v)$ is also zero.

Let $\sum_i n_i(P_i)$ be a zero-cycle of degree $d$ on the generic fiber. Then each point $P_i$, say of degree $d_i$, gives rise to an irreducible subvariety $E_i\subset X$ such that $E_i\to Z$ is finite surjective of degree $d_i$. By shrinking $Z_0$ if necessary, we can assume $(E_i)_{Z_0}:=E_i\times_Z {Z_0}\to Z_0$ is finite \'etale of degree $d_i$ for each $i$. Since each point in $X(k_v)$ has an analytic neighborhood over $k_v$ for any place $v$, we have $X_{Z_0}(k_v)\neq\emptyset$, so $X_{Z_0}(\A)\neq\emptyset$ by properness of the fibers.

Let $\{P_v\}\in X(\A)^B$. By left-continuity of the Brauer--Manin pairing, we may deform $P_v$ so that $\{P_v\}\in X_{Z_0}(\A)\cap X(\A)^B$. By the implicit function theorem, we can find a $v$-adic open set $U_v\subset Z_0$, containing $\pi(P_v)$, where there exists a local section $\rho_v\colon U_v\to X$, i.e., $\pi\circ\rho_v=\id$ and $\rho_v(\pi(P_v))=P_v$. By weak approximation, we can find a rational point $Q\in Z_0(k)$ such that for each $v\in S$, $Q$ is $v$-adically close enough to $\pi(P_v)$ so that $Q\in U_v$ and $\calB(\rho_v(Q))=\calB(P_v)$ for any $\calB\in H$.

Let $M_i\in \CH_0(X)$ be the zero-cycle of degree $d_i$ corresponding to the subscheme $E_i\times_{Z} Q\subset X$. By construction, we can also consider $M_i\in \CH_0(F)$ where $F:=\pi^{-1}(Q)$ is the fiber above $Q$. Let $M=\sum_i n_iM_i$, which is a zero-cycle of degree $d$, and denote its image in $\CH_0(F_{k_v})$ by $M_v$ for any $v\in\Omega$.

Define the point $\{R_v\}\in X(\A)$ as follows. For each $v\notin S$, choose $R_v$ to be any point in $X(k_v)$. For each $v\in S$, first set $O_v:=\rho_v(Q)\in F(k_v)$. Let $n\in \Z$ be such that $nd\equiv 1\mod {|H|}$ and define the zero-cycle
$$D:=nM_v-(nd-1)(O_v).$$
By the assumption that $F$ satisfies (ZC), there is a point in $R_v\in F(k_v)$ such that $(R_v)=D$ in $\CH_0(F_{k_v})$.

We now show that $\{R_v\}\in X(\A)^H$. Fix some $\calA\in H\cap B$ or $\calA\in H[d^\perp]$. We have
$$\sum_{v\in\Omega}\inv_v\calA (M_v)=0$$
since $M_v$ is the image of the $k$-rational zero-cycle $M$. Recall that by assumption on $S$, the map $\calA(-)\colon \CH_0(X_{k_v}) \to\Br(k_v)$ is zero for $v\notin S$. Hence
\begin{align*}
\sum_{v\in \Omega} \inv_v\calA (R_v)&=\sum_{v\in S} \inv_v\calA (R_v)\\
&=\sum_{v\in S} n\inv_v\calA (M_v)+\sum_{v\in S} (1-nd)\inv_v\calA (O_v)\\
&=\sum_{v\in \Omega} n\inv_v\calA (M_v)+\sum_{v\in \Omega} (1-nd)\inv_v\calA (P_v)\\
&=0
\end{align*}
where the last equality follows from $1-nd\equiv 0\mod |H[d^\perp]|$ if $\calA\in H[d^\perp]$ and from $\{P_v\}\in X(\A)^B$ if $\calA\in B\cap H$. Hence $\{R_v\}\in X(\A)^H$.
\end{proof}

\section{Proof of Theorem \ref{main2}}
The proof follows similar strategy as \S\ref{sec:mainth}; the key ingredient is \cite[Lemma 4.6]{cv17} whose use was suggested to us by Skorobogatov. We use this to prove a slight variant of \cite[Corollary 4.3]{cv17}.

\begin{lem}\label{abvar}
Let $A$ be an abelian variety over a number field $k$. Let $\psi\colon Y\to A$ be a $d$-covering. Let $H\subset \Br(Y)$ be a finite subgroup. Suppose $S\subset\Omega$ is a finite set of places such that the group $\CH_0(Y_{k_v})$ pairs trivially with $H$ for all $v\notin S$.

If there exists $\{R_v\}\in \prod_{v\in S} Y(k_v)$ such that
$$\sum_{v\in S} \inv_v \calA(R_v)=0$$
for all $\calA\in H[d^\infty]$, then there exists $\{P_v\}\in \prod_{v\in S}Y(k_v)$ such that
$$\sum_{v\in S} \inv_v \calA(P_v)=0.$$
for all $\calA\in H$.
\end{lem}

\begin{proof}
By \cite[Lemma 4.6, A.2]{cv17}, there is a map $\rho\colon Y\to Y$ such that
$$\rho^*\colon \Br(Y)\to\Br(Y)$$
is the identity on $H[d^\infty]$ and $\rho^*H[d^\perp]\subset\Br_0(Y)=\Br(k)$. Let $P_v=\rho(R_v)$. For any $\calA\in H[d^\infty]$,
$$\sum_{v\in S}\calA(P_v)=\sum_{v\in S}\rho^*\calA(R_v)=\sum_{v\in S}\calA(R_v)=0.$$
Now let $\calA\in H[d^\perp]$ and $\alpha=\rho^*\calA\in \Br(k)$. For any $v\notin S$ and $M\in \CH_0(Y_{k_v})$, we have $\alpha(M)=\calA(\rho_*M)=0$ by assumption. On the other hand, $\alpha(M)=\deg(M)\alpha_v$ where $\alpha_v$ is the image of $\alpha$ under $\Br(k)\to \Br(k_v)$. Since $Y_{k_v}$ has a zero-cycle of degree $d^{2g}$, namely $\psi^{-1}(0)$, it follows $\alpha_v\in \Br(k_v)[d^{2g}]$ which shows $\alpha_v=0$. Hence
$$\sum_{v\in S}\calA(P_v)=\sum_{v\in S}\alpha_v=\sum_{v\in\Omega}\alpha_v=0.$$
Thus $\{P_v\}$ is orthogonal to $H[d^\infty]+H[d^\perp]=H$.
\end{proof}

\begin{proof}[Proof of Theorem \ref{main2}] Let $H\subset B+\Br(X)[d^\perp]$ be any finite subgroup. We must show that $X(\A)^H\neq\emptyset$. Let $S\subset\Omega$ be a finite set of places such that $H$ pairs trivially with $\CH_0(X_{k_v})$ for all $v\notin S$ (see \S\ref{sec:mainth}). There is a Zariski open set $U\subset Z$ such that $A$ has a model $\calA$ over $U$ and for each point $P\in U$, the fiber $X_P$ is a $k(P)$-torsor under the abelian variety $\calA_P$. Let $\{T_v\}\in X(\A)^B$. As in \S\ref{sec:mainth}, we can find a point $Q\in U(k)$ that is $v$-adically close to $\pi(T_v)$ so that $F:=\pi^{-1}(Q)$ has a $k_v$-point $R_v$ with $\calA(R_v)=\calA(T_v)$ for each $v\in S$ and $\calA\in H$. In particular,
$$\sum_{v\in S} \inv_v\calA(R_v)=\sum_{v\in S} \inv_v\calA(T_v)=\sum_{v\in \Omega} \inv_v\calA(T_v)=0$$
for any $\calA\in H[d^\infty]\subset B$. An application of Lemma \ref{abvar} then gives a point $\{P_v\}\in\prod_{v\in S} F(k_v)$ such that
$$\sum_{v\in S} \inv_v\calA(P_v)=0$$
for all $\calA\in H$. Setting $P_v$ to be any point on $X(k_v)$ for $v\notin S$, we conclude $\{P_v\}\in X(\A)^H$.\end{proof}

\section{Applications}\label{app}

Before we prove the corollaries listed in $\S1$, we first give examples of varieties satisfying property (ZC).

\begin{example}\label{propp} Any $k$-rational variety satisfies property (ZC). In fact, any class of varieties that become $K$-rational once they have a $K$-point satisfies (ZC) such as
\begin{enumerate}
\item {\it Quadrics.} In this case we always have a zero-cycle of degree 2. Hence the theorem implies $X(\A)\neq\emptyset\implies X(\A)^{\Br[2^\perp]}\neq\emptyset$ when the fibers are quadrics.
\item {\it Severi--Brauer varieties.} In this case one can also replace $d$ in Theorem \ref{main} by the period of the generic fiber, since index and period have the same prime divisors.
\item {\it Del Pezzo surfaces of degree $d\geq5$.}
\end{enumerate}
\end{example}

\begin{example}\label{ppp3}
Let $X$ be a del Pezzo surface of degree 4. If $X(k)\neq\emptyset$, then there exists a point $P\in X(k)$ not lying on any of the lines on $X$. One can blow up at $P$ to get a cubic surface $\widetilde{X}$ containing a line $L$ defined over $k$. The projection from $L$ defines a rational map $\widetilde{X}\dashrightarrow \P^1$ which can be extended to a morphism where the fibers are conics. Since there are exactly 10 lines on $\widetilde{X}$ intersecting $L$, there are 5 degenerate fibers. By \cite{ctc79} $\widetilde{X}$ satisfies (ZC), which implies $X$ satisfies (ZC) by birational invariance of $CH_0(X)$ among smooth projective varieties. Since $X$ can be expressed as intersection of two quadrics in $\P^4$, its index divides 4.
\end{example}

\begin{example}\label{gen1}
Let $X$ be a smooth curve of genus 1. Then $X$ satsifies (ZC) by the Riemann-Roch theorem for curves. This is not true for higher dimensional abelian varieties, but one can apply Theorem \ref{main2} instead in this case.
\end{example}

\begin{rmk}By taking $Z=\Spec k$ in Theorem \ref{main}, we get that any of the varieties $X$ listed above satisfy $X(\A)\neq\emptyset\implies X(\A)^{\Br(X)[d^\perp]}\neq\emptyset$ where $d$ is the index of $X$. If $X$ is a smooth projective curve of genus 1, then one can take $d$ to be the order of $[X]\in \H^1(k,J)$ where $J$ is the Jacobian of $X$.
\end{rmk}

\subsection{Del Pezzo Surfaces}

Let $X$ be a del Pezzo surface of degree $2$ over a number field $k$. See \cite[Theorem 4.1]{corndp2} for a complete list of isomorphism classes of $\Br(X)/\Br_0(X)$. In particular, $\Br(X)/\Br_0(X)$ has exponent either 2, 3, or 4, and if it has exponent 3, then it is isomorphic to either $\Z/3\Z$ or $(\Z/3\Z)^2$.

\begin{proof}[Proof of Corollary \ref{dP2}]
We can express $X$ as a smooth hypersurface of degree $4$ inside the weighted projective space $\Proj k[w,x,y,z]=\P(2,1,1,1)$. Since $\Char k\neq2$, $X$ has an equation of the form
$$w^2=f(x,y,z),$$
where $f(x,y,z)$ is a homogeneous polynomial of degree $4$. Suppose $\Br(X)/\Br(k)\isom \Z/3\Z$ or $(\Z/3\Z)^2$. Let $\pi\colon X\dashrightarrow \P^1$ be the rational map given by $[w:x:y:z]\mapsto[y:z]$. The locus of indeterminancy is given by $Z:=\{[w:x:y:z]\in X \mid y=z=0\}$. Let $\widetilde{X}$ be the blow up of $X$ along $Z$ giving the following diagram
$$
\xymatrixcolsep{2pc}\
\xymatrix{
\widetilde{X} \ar[dr]^{\widetilde{\pi}} \ar[d]^\beta \\\
X\ar@{.>}[r]^\pi & \mathbb{P}^1}
$$
Then $\widetilde{\pi}$ is a genus 1 fibration, i.e., almost all fibers are smooth projective curves of genus 1. The generic fiber clearly has points of degree $2$ over $k(\P^1)$. Hence by Example \ref{gen1} and Theorem \ref{main}, the subgroup $\Br(\widetilde{X})[2^\perp]$ does not obstruct the Hasse principle. Since $\beta^*\Br(X)[2^\perp]\subseteq \Br(\widetilde{X})[2^\perp]$, if $\{P_v\}\in \widetilde{X}(\A)^{\Br(\widetilde{X})[2^\perp]}$, then $\beta(\{P_v\})\in X(\A)^{\Br(X)[2^\perp]}$. The statement of the corollary follows from the equality $(\Br(X)/\Br_0(X))[2^\perp]=\Br(X)/\Br_0(X)$.
\end{proof}

\begin{rmk} As genus one curves are torsors under their Jacobian, one can also use Theorem \ref{main2} to prove the above result.
\end{rmk}

\begin{example}\label{example}
For any cubic surface $S\subset\P^3$ with 3 torsion in $\Br(S)/\Br_0(S)$, one can blow up a point on $S$ to obtain a del Pezzo surface $X$ of degree 2 with 3 torsion in $\Br(X)/\Br_0(X)$. However, such an $X$ will always have a rational point. An example of a del Pezzo surface over $\Q$ with 3 torsion in Brauer group which does not arise as a blow up of a cubic surface is the following (provided to us by Elsenhans)
\begin{align*}
w^2= &\ 8x^4 + 16x^3y - 16x^3z + 21x^2y^2 - 30x^2yz + 21x^2z^2 + 14xy^3 - 24xy^2z +\\&\ 18xyz^2 - 
    16xz^3 + 7y^4 - 22y^3z + 9y^2z^2 + 4yz^3 + 2z^4.
\end{align*}
Indeed, the Galois action on $\Pic(\Xbar)$ is given by an order 12 subgroup of the Weil group $W(E_7)$, and for any order 12 subgroup $G\subset W(E_6)$, the Galois cohomology group $\H^1(G,\Z^7)$ has no elements of order 3. One can check that the above surface has local points everywhere and hence no Brauer--Manin obstruction. Indeed $[w:x:y:z]=[0:0:1:1]$ is a rational point.
\end{example}

\subsection{K3 Surfaces}

\begin{proof}[Proof of Corollary \ref{corodd}] Suppose $X(\A)\neq\emptyset$. Then the quadric $Q$ defined by
$$ar^2+bs^2+cu^2+dv^2=0$$
is also everywhere locally soluble, and thus has a $k$-point by the Hasse--Minkowski theorem. The condition that $abcd$ is a square implies that there is a ruling $\pi\colon Q\to\P^1$ over $k$. Under the natural map $\rho\colon X\to Q$, the pullback of a line in the fiber of $\pi$ is a curve of genus 1, making $\pi\circ\rho$ a genus 1 fibration (See \cite{swinn2000}). The generic fiber $X_{k(\P^1)}$ is a degree 8 cover of $Q_{k(\P^1)}$. A choice of a section $s\colon\P^1\to Q$ gives a point $\Spec k(\P^1)\to Q_{k(\P^1)}$. The projection of $\Spec k(\P^1)\times_{Q_{k(\P^1)}} X_{k(\P^1)}$ to $X_{k(\P^1)}$ gives a degree 8 zero-cycle on $X_{k(\P^1)}$. An application of Theorem \ref{main} to $\pi\circ\rho\colon X\to \P^1$ finishes the proof.
\end{proof}

\subsubsection{Construction of the Kummer varieties in Corollary \ref{kum}}\label{constkum} Let $A$ be an abelian variety of dimension $\geq2$ defined over $k$. Let $\psi\colon Y\to A$ be a 2-coverings of $A$. The antipodal involution on $A$ induces an involution $\sigma$ on $Y$. Let $Y'\to Y$ be the blow up along $f^{-1}(0)$. Then $\sigma$ extends to $Y'$; the quotient $X=Y'/\sigma$ is smooth and is called the Kummer variety $Kum(Y)$ attached to $Y$ (see \cite{sz2016} for more details).

\begin{proof}[Proof of Corollary \ref{kum}]Let $\{P_v\}\in X(\A)^B$. It suffices to show that $X(\A)^H\neq\emptyset$ for any finite subgroup $H\subset B+\Br(X)[2^\perp]$. Let $S\subset \Omega$ be a finite set of places such that $H$ pairs trivially with $\CH_0(X_{k_v})$ for any $v\notin S$ (see first paragraph of proof of Theorem \ref{main}). For each $v\in S$, there is a class $\alpha_v\in \H^1(k_v, \mu_2)=k_v^\times/k_v^{\times2}$ such that $P_v$ lifts to a $k_v$-point on the twist $Y'_{\alpha_v}$. By weak approximation, there is a class $\alpha\in \H^1(k, \mu_2)$ which restricts to $\alpha_v$ for each $v\in S$. Hence the twist $f \colon Y'_{\alpha}\to X$ contains a point in $\{Q_v\}_{v\in S}$ mapping to $\{P_v\}_{v\in S}$. Since $Y'_\alpha$ and $Y_\alpha$ are birational, by abuse of notation, we can identify $f^*H$ as a subgroup of $\Br(Y_\alpha)$ and $\{Q_v\}_{v\in S}$ as a point on $Y_\alpha$. Then $$\sum_{v\in S} \inv_v\calA(Q_v)=0$$
for any $\calA\in f^*(H\cap B)=f^*H[2^\infty]$. An application of Lemma \ref{abvar} shows there is a point $\{Q'_v\}_{v\in S}$ such that
$$\sum_{v\in S} \inv_v\calA(Q'_v)=0$$
for any $\calA\in f^*H$. Take any point in the preimage in $Y'_\alpha$ and map it down to $X$ to obtain a point $\{P'_v\}_{v\in S}$. Set $P'_v$ to be any point on $X(k_v)$ for $v\notin S$ to get $\{P'_v\}\in X(\A)^H$.
\end{proof}

\bibliographystyle{alpha}
\bibliography{hiro}

\end{document}